\documentclass[twoside,11pt]{article}%
 \usepackage{graphicx,wrapfig}
\usepackage{graphicx,subfigure}
\usepackage{fancyhdr,graphicx}
\usepackage{multicol}
\usepackage[center]{caption2}
\usepackage{epstopdf}
\usepackage{bm}
\usepackage{latexsym}
\usepackage{amsmath,amssymb,epsfig}
\usepackage{amsthm,enumerate,verbatim}
\usepackage{amsfonts}
\usepackage{color}
\usepackage{mathrsfs}
\usepackage{cite}
\providecommand{\U}[1]{\protect\rule{.1in}{.1in}}
\providecommand{\U}[1]{\protect\rule{.1in}{.1in}}
\newcommand{\BE}{\begin{equation}}
\newcommand{\EE}{\end{equation}}

\numberwithin{equation}{section}
\newtheorem{proposition}{Proposition}[section]
\newtheorem{theorem}[proposition]{Theorem}
\newtheorem{lemma}[proposition]{Lemma}

\newtheorem{remark}[proposition]{Remark}
\newtheorem{example}[proposition]{Example}

\def\dfrac{\displaystyle\frac}

\begin{document}

 \title{\bf  On a second order scheme for space fractional
diffusion equations with variable coefficients}
 \author{Seakweng Vong\thanks{Email: swvong@umac.mo.} \quad Pin Lyu\thanks{
 Email: lyupin1991@163.com.}\\
{\footnotesize  \textit{Department of Mathematics, University of Macau, Avenida da Universidade, Macau, China}}}
\date{}
 \maketitle
\date{}

\begin{abstract}
 We study a second order scheme for spatial fractional differential equations with variable coefficients.
 Previous results mainly
  concentrate on equations with diffusion coefficients that are proportional to each
  other. In this paper, by further study on the generating function of the discretization matrix,
  second order convergence of the  scheme is proved for diffusion coefficients satisfying a certain condition
  but are not necessary to be proportional.
  The theoretical results are justified by numerical tests.
\end{abstract}


\section{Introduction}
 In recent studies, people find that problems from different areas such as chemistry, physics, engineering and medical science can be
 modeled by fractional derivatives. The study of fractional differential equations (FDEs) give rise to an active research direction both theoretically and numerically. However, due to the nonlocal property of the fractional differential operators, many classical methods for differential equations cannot be applied directly, it makes this research direction very challenging.
 In this paper, we concentrate on numerical aspect of solving fractional differential equations. More specifically,
 we consider high order finite difference  scheme for  differential equation with spatial fractional derivatives in the following form:
 \begin{equation}\label{fde}
\begin{array}{l}
 \displaystyle{\frac{\partial u(x, t)}{\partial t}=
 d_+(x)\,{_{x_L}}D_x^{\alpha}u(x,t)+d_-(x)\,{_x}D_{x_R}^{\alpha}u(x,t)+f(x,t)}, \quad x\in(x_L,x_R),~t\in(0,T],\\
 \vspace{2mm}
 u(a,t)=u(b,t)=0, \qquad t\in[0,T],\\
 u(x,0)=\varphi(x),\qquad x\in[x_L,x_R],
\end{array}
\end{equation}
where $\alpha\in(1,2)$,
and $d_+(x)$, $d_-(x)$ are nonnegative functions. The
 ntotations
 ${_{x_L}}D_x^{\alpha}u(x)$ and $\,{_x}D_{x_R}^{\alpha}u(x)$ denote the $\alpha$-order left and right Riemann-Liouville
fractional derivatives of $u(x)$, respectively, and they are defined
as
$$
\begin{array}{l}\vspace{2mm}
\displaystyle{\,{_{x_L}}D_x^{\alpha}u(x)=\dfrac{1}{\Gamma(2-\alpha)}\dfrac{\partial^2}{\partial
x^2}\int_{x_L}^x \dfrac{u(\xi)}{(x-\xi)^{\alpha-1}}\,d\xi},\\
\displaystyle{\,{_x}D_{x_R}^{\alpha}u(x)=\dfrac{1}{\Gamma(2-\alpha)}\frac{\partial^2}{\partial
x^2}\int_x^{x_R} \frac{u(\xi)}{(\xi-x)^{\alpha-1}}\,d\xi},
\end{array}
$$
 with $\Gamma(\cdot)$ being the gamma function.

 Quite a lot of progress have been made for the study of space FDEs
 in past decades \cite{Deng_Secondorder,Deng_Fourthorder,DengCCP,DengSecondorder2,VongNACO,DengJSC,SousaANM,Yuste,SousaCMA,DengAMM,DengNPDE,Hao-Sun,VongNA,Celik,zhao_siamJSC,DengSIAMjsc,LiuF1,LiuF2,LiuAMC2007,Lin-Jin,LiuF_NA,Pan_NA}.
 High order difference approximations are very popular in solving both space FDEs \cite{LiuF2,Deng_Secondorder,Deng_Fourthorder,DengCCP,DengSecondorder2,VongNACO,DengJSC,SousaANM,Yuste,SousaCMA,DengAMM,DengNPDE,Hao-Sun,VongNA,Celik,zhao_siamJSC,DengSIAMjsc} and time FDEs \cite{SunJCP1,
 VongJCP,SunJSC1,VongIJCM,LiaoJSC,LyuVong}. Interested reader can refer to \cite{Deng_Secondorder,Deng_Fourthorder,DengCCP,DengSecondorder2,VongNACO,DengJSC,SousaANM,Yuste} and the references therein for tracing  the whole picture that are closely related to this paper.
 First order approximation of fractional derivatives by Gr\"{u}nwald-Letnikov formula was introduced in \cite{Podlubny}.
 Then, by considering the  weighted and shifted
 Gr\"{u}nwald-Letnikov difference (WSGD)
 formulas and weighted and shifted Lubich
 formulas, second and fourth order approximation were developed in  \cite{Deng_Secondorder} and \cite{Deng_Fourthorder}, respectively.
 The corresponding schemes were shown to be unconditionally stable and convergent.
 We remark that the above results on high order schemes are restricted to cases
 that the diffusion coefficients are proportional to each other \cite{Deng_Fourthorder,DengCCP,DengSecondorder2,VongNACO} or, as special case, are constants \cite{Deng_Secondorder,DengSecondorder2,DengNPDE}.
One of the usually used ways for the theoretical analysis of this kind of difference scheme is to show that the corresponding differentiation matrices have eigenvalues with negative real parts, see \cite{Deng_Secondorder} and \cite{Deng_Fourthorder} for examples.
 To our knowledge, if the whole range of $\alpha\in(1,2)$ is under consideration,
 convergence of schemes to \eqref{fde} with general variable coefficients can be established only for (spatial) first order discretization due to its differentiation matrix is diagonally dominant, but
this property fails for higher order differentiation matrices with general diffusion coefficients, see Remark \ref{exampleD} in the next section.

 In this paper, we focus on a second order WSGD approximation, which was first introduced in \cite{Deng_Secondorder}, for space fractional diffusion equations with variable coefficients. Our aim is to establish the theoretical analysis for this type of second order scheme by proving its differential matrix has eigenvalues with nonpositive real parts. We show that if the diffusion coefficients satisfy a certain condition,
 then the symmetric part  of the differentiation matrix $A$ is negative semidefinite which implies that the real parts of all eigenvalues of $A$ are nonpositive \cite{MatrixAnalysisHorn}. The main idea to achieve this is based on further analysis of the generating function of the differentiation matrix.
 Consequently, the analysis for the second order scheme can be carried out.
We remark that the generating function can be a good tool to analyze some properties of the differential matrix which has Toeplitz structure. This approach has also been adopted in our recent work \cite{VongShi}.
Here we remark that, since the discretization coefficients of high order difference operators (at least second order ones, and not only the WSGD operator) do not have the good property possessed by the first order one, the corresponding scheme  may not work as expected for general variable coefficients.
This fact is reflected in subsection \ref{counterexample}, where some examples that do not satisfy our proposed condition \eqref{cond} are tested and numerical results indicate that the second order scheme fails to solve these examples with reliable behavior. Therefore it is interesting and necessary to derive some conditions under which the scheme works properly.

 This paper is organized as follows. Some preliminary concepts about the one-dimensional second order scheme are reviewed in section \ref{SecondorderSection}. Section \ref{1dsc} presents our main result, a condition is derived and under which the one-dimensional scheme is shown to be stable and convergent with respect to $L^2$ norm by using discrete energy method. In section \ref{secondorder2D}, an alternative condition is introduced and the second order ADI scheme for two-dimensional problem is considered. In section \ref{numerical}, some numerical examples are carried out to justify our theoretical analysis. A brief conclusion is followed in the last section.

\section{Preliminaries}\label{SecondorderSection}
Let $M$, $N$ be positive integers, and let $h = (x_R-x_L)/M$, $\tau =
T/N$ be the space step and time step, respectively. We define
spatial and temporal partitions: $x_i = x_L + i h$ for $i=
0,1,\ldots,M$; $t_n = n\tau$ for $n = 0,1,\ldots,N$. In
\cite{{Deng_Secondorder}}, Tian et al. introduced the WSGD formulas
and considered a class of second order discretization formula with the following form:
\begin{equation}\label{FD2nd}
\begin{array}{l} \vspace{2mm}
\displaystyle{{_{x_L}}D_x^{\alpha}u(x_i)=\frac{1}{h^{\alpha}}\sum_{k=0}^{i+1}w_{k}^{(\alpha)}u(x_{i-k+1})+{\cal O}(h^2)},\\
\displaystyle{{_x}D_{x_R}^{\alpha}u(x_i)=\frac{1}{h^{\alpha}}\sum_{k=0}^{M-i+1}w_{k}^{(\alpha)}u(x_{i+k-1}
)+{\cal O}(h^2)},
\end{array}
\end{equation}
which is under the smooth assumptions $u\in L^1(\mathbb{R})$; and $_{-\infty}D_x^{\alpha+2}u$, $_xD_\infty^{\alpha+2}u$ and their Fourier transform belong to $L^1(\mathbb{R})$. The coefficients $w_{k}^{(\alpha)}$ depend on a pair of shifting parameters $(p,q)$. In this paper, we concentrate on
\begin{equation}\label{w_{k}(1,0)}
 w_0^{(\alpha)}=\frac{\alpha}{2}g_0^{(\alpha)},
 ~w_k^{(\alpha)}=\frac{\alpha}{2}g_k^{(\alpha)}+\frac{2-\alpha}{2}g_{k-1}^{(\alpha)}~\mbox{  for  }~k\geq1,
\end{equation}
which corresponds to $(p,q)=(1,0)$ (\cite{Deng_Secondorder})
, where  $g_k^{(\alpha)}$ are the coefficients of the power series of the function $(1-z)^\alpha$, and they can be evaluated recursively as
\begin{align}\label{gkalpha}
g_0^{(\alpha)}=1,\quad g_k^{(\alpha)}=\Big(1-\frac{\alpha+1}{k} \Big)g_{k-1}^{(\alpha)},~\mbox{ for }~k=1,2,\ldots.
\end{align}

Let $u^{n}_i$ be the numerical approximation of $u(x_i,t_n)$, and $d_{+,i}=d_+(x_i)$, $d_{-,i}=d_-(x_i)$, $\varphi_i=\varphi(x_i)$, $f^{n+\frac{1}{2}}_i=f(x_i,t_{n+\frac{1}{2}})$, where $t_{n+\frac{1}{2}}=\frac{1}{2}(t_n+t_{n+1})$, $i=1, 2,\ldots, M-1$, $n=0,1,\ldots,N-1$.
Then applying the Crank-Nicolson technique and approximations
(\ref{FD2nd}) to the time derivative and the space fractional derivatives of (\ref{fde}) respectively, we get
\begin{align}\nonumber
 \frac{u_{i}^{n+1}-u_{i}^{n}}{\tau}=
 &\frac{1}{2h^{\alpha}}\left(d_{+,i}\sum_{k=0}^{{i}}w_{k}^{(\alpha)}u_{i-k+1}^{n}
 +d_{+,i}\sum_{k=0}^{{i}}w_{k}^{(\alpha)}u_{i-k+1}^{n+1}\right.\\\label{CN}
 &\left.+d_{-,i}\sum_{k=0}^{{M-i}}w_{k}^{(\alpha)}u_{i+k-1}^{n}
 +d_{-,i}\sum_{k=0}^{{M-i}}w_{k}^{(\alpha)}u_{i+k-1}^{n+1}\right)+f_{i}^{n+\frac{1}{2}}+R_i^{n+\frac12},
\end{align}
where $R_i^{n+\frac12}\leq c_1(\tau^2+h^2)$ for a positive constant $c_1$.

Multiplying \eqref{CN} by $\tau$ and omitting the small term $\tau R_i^{n+\frac12}$, we obtain the following equation:
\begin{align}\nonumber
u_{i}^{n+1}&-\frac{\tau}{2h^{\alpha}}\left(d_{+,i}\sum_{k=0}^{{i}}w_{k}^{(\alpha)}u_{i-k+1}^{n+1}
+d_{-,i}\sum_{k=0}^{{M-i}}w_{k}^{(\alpha)}u_{i+k-1}^{n+1}\right)
\\&=u_{i}^{n}+
 \frac{\tau}{2h^{\alpha}}\left(d_{+,i}\sum_{k=0}^{{i}}w_{k}^{(\alpha)}u_{i-k+1}^{n}
 +d_{-,i}\sum_{k=0}^{{M-i}}w_{k}^{(\alpha)}u_{i+k-1}^{n}\right)
 +\tau f_{i}^{n+\frac{1}{2}},\label{CN-1}
\end{align}

Denote $\nu_{\tau,h,\alpha}=\frac{\tau}{2 h^{\alpha}}$, $u^n=[u_1^n,u_2^n,\ldots,u_{M-1}^n]^T$ and
$f^{n+\frac{1}{2}}=[f_1^{n+\frac{1}{2}},f_2^{n+\frac{1}{2}},\ldots,f_{M-1}^{n+\frac{1}{2}}]^T$.
 Let
\begin{eqnarray}\label{W}
W_\alpha=\left[
    \begin{array}{ccccc}
       w_1^{(\alpha)}& w_0^{(\alpha)} & 0 & \cdots &  0 \\
       w_2^{(\alpha)} & w_1^{(\alpha)} & w_0^{(\alpha)} & \ddots & \vdots \\
      \vdots & w_2^{(\alpha)} & w_1^{(\alpha)} & \ddots & 0\\
      \vdots & \ddots & \ddots & \ddots & w_0^{(\alpha)} \\
      w_{M-1}^{(\alpha)}  & \cdots  & \cdots & w_2^{(\alpha)} & w_1^{(\alpha)} \\
    \end{array}
\right],
\end{eqnarray}
 where  $\{w_k^{(\alpha)}\}_{k=0}^{M-1}$ are the coefficients given in \eqref{w_{k}(1,0)}. Denote $A=D_{+}W_\alpha+D_{-}W_\alpha^{T}$, where matrices $D_+$, $D_-$ are defined as
\begin{eqnarray}\label{D}
D_+=\left[
    \begin{array}{cccc}
       d_{+,1}& & & \\
       & d_{+,2}&& \\
       && \ddots & \\
       &&&d_{+,M-1}\\
    \end{array}
\right],
\quad
D_-=\left[
    \begin{array}{cccc}
       d_{-,1}& & & \\
       & d_{-,2}&& \\
       && \ddots & \\
       &&&d_{-,M-1}\\
    \end{array}
\right].
\end{eqnarray}

Then the scheme \eqref{CN-1} can be expressed in the following matrix form
\begin{equation}\label{Matrixform}
\left(I-\nu_{\tau,h,\alpha}A \right)u^{n+1}=\left(I+\nu_{\tau,h,\alpha}A \right)u^{n}+\tau f^{n+\frac{1}{2}},
\quad n=0,1,\ldots,N-1,
\end{equation}
with $I$ being the identity matrix.

 Before moving to the analysis in the next subsection, we first point out the following remark,
 which gives one of the motivations of this paper.

\begin{remark}\label{exampleD}
 For general variable coefficients $d_+(x)$ and $d_-(x)$,
 the eigenvalues of the differentiation matrix $A=D_+W_\alpha+D_-W_\alpha^T$
 may have positive real part
 while $\alpha$ close to $1$. For example,
 we consider $M=4$,
\begin{eqnarray}\nonumber
D_+=\left[
    \begin{array}{ccc}
       3& &  \\
       &1/2& \\
       &&\sqrt{3}\\
    \end{array}
\right],
\qquad \mbox{and}\quad
D_-=\left[
    \begin{array}{ccc}
       1& &  \\
       &1& \\
       &&\sqrt{3}\\
    \end{array}
\right].
\end{eqnarray}
 when $\alpha=1.1$, 
 one can check that  the real part of the eigenvalues of the matrix $A$ are:
$$\Re(\lambda_1)=0.1801>0, \quad \Re(\lambda_2)=-1.0706,\quad \Re(\lambda_3)=-0.499.$$
 Therefore $A$  fails to fulfill the key property (see \cite{Deng_Secondorder} and  \cite{Deng_Fourthorder} for example)
 for establishing stability of the corresponding schemes.
 We also note that there are similar examples \cite{VongNACO} for another second order discretization developed in \cite{SousaANM,SousaCMA}.
 From all these examples, it seems not easy to theoretically consider high order schemes for \eqref{fde} in general
 when the diffusion coefficients  are not propositional to each
  other.

\end{remark}

\section{Theoretical Analysis}\label{1dsc}
First, we introduce the background knowledge of Toeplitz matrix and the generating function. A matrix that has the following form
 \begin{eqnarray}\nonumber
T_n=\left[
    \begin{array}{ccccc}
       t_0& t_{-1} & \cdots &  t_{2-n} & t_{1-n} \\
       t_1& t_0 & t_{-1} & \cdots &  t_{2-n}  \\
      \vdots & t_1 & t_0 & \ddots & \vdots\\
      t_{n-2}& \cdots & \ddots & \ddots &  t_{-1}\\
      t_{n-1}& t_{n-2}& \cdots &  t_1   & t_0 \\
    \end{array}
\right],
\end{eqnarray}
is called a $n$-by-$n$ Toeplitz matrix \cite{Jin}, where $t_{ij}=t_{i-j}$, i.e., the entries of $T_n$ are constant along each diagonal. If the diagonals
of $T_n$ are the Fourier coefficients of a function $f$:
$$t_k(f)=\frac1{2\pi}\int_{-\pi}^{\pi}f(\eta)e^{-{\bf i}k\eta}d \eta,$$
then the function $f$ is called the generating function of $T_n$.

 Noticing Remark \ref{exampleD}, we will study some conditions, under which
 the  scheme for \eqref{fde} can be analyzed.
 The condition depends on a constant related to some geometric property for the generating function
 \cite{Jin} of the Toeplitz matrix $W_\alpha$. Namely, we have the following:
\begin{theorem}\label{generatingfunction_rario}
 Let $g(\alpha,x)$ be the generating function of  $W_\alpha$ defined in \eqref{W}. 
 Then it holds that
$$\varsigma_\alpha\triangleq\min_{x}\frac{\Re[-g(\alpha,x)]}{|g(\alpha,x)|}=|\cos(\frac{\alpha}{2}\pi)|,$$
where $\Re[g(\alpha,x)]$ denotes the real part of $g(\alpha,x)$.
\end{theorem}
\begin{proof}
For simplicity of presentation, we use $\Re[g]$ to denote $\Re[g(\alpha,x)]$. Referring to Theorem 2 in \cite{Deng_Secondorder}, we have $\Re[-g]\ge0$ ($1<\alpha<2$). 
For $x\in[0,\pi]$, with the coefficients $w_k^{(\alpha)}$ given by \eqref{w_{k}(1,0)}, we have
\begin{align*}
g(\alpha,x)=&\sum_{k=0}^{\infty}w_k^{(\alpha)}e^{{\bf i}(k-1)x}=\frac{\alpha}{2}e^{-{\bf i}x}\sum_{k=0}^{\infty}
g_k^{(\alpha)}e^{{\bf i}kx}+\frac{2-\alpha}{2}\sum_{k=0}^{\infty}g_k^{(\alpha)}e^{{\bf i}kx}\\
=&\frac{\alpha}{2}e^{-{\bf i}x}(1-e^{{\bf i}x})^\alpha+\frac{2-\alpha}{2}(1-e^{{\bf i}x})^\alpha\\
=&\big(2\sin(\frac{x}{2}) \big)^\alpha \Big[\frac{\alpha}{2}\cos\big(\frac{\alpha}{2}(x-\pi)-x\big)
+\frac{2-\alpha}{2}\cos\big(\frac{\alpha}{2}(x-\pi)\big) \\
&+{\bf i}\Big( \frac{\alpha}{2}\sin\big(\frac{\alpha}{2}(x-\pi)-x\big)
+\frac{2-\alpha}{2}\sin\big(\frac{\alpha}{2}(x-\pi)\big) \Big) \Big].
\end{align*}
Let $\Im[g]$ be the imaginary part of $g(\alpha,x)$, we have
$$\frac{\Im[g]}{\Re[g]}=\frac{\frac{\alpha}{2}\sin\big(\frac{\alpha}{2}(x-\pi)-x\big)
+\frac{2-\alpha}{2}\sin\big(\frac{\alpha}{2}(x-\pi)\big)}{\frac{\alpha}{2}\cos\big(\frac{\alpha}{2}(x-\pi)-x\big)
+\frac{2-\alpha}{2}\cos\big(\frac{\alpha}{2}(x-\pi)\big)},$$
and then
$$\bigg(\frac{\Im[g]}{\Re[g]}\bigg)_x=\frac{\frac{\alpha}{2}(1-\frac{\alpha}{2})(\alpha-1)\big(\cos(x)-1\big)}{\Big[\frac{\alpha}{2}\cos\big(\frac{\alpha}{2}(x-\pi)-x\big)
+\frac{2-\alpha}{2}\cos\big(\frac{\alpha}{2}(x-\pi)\big)
\Big]^2}\leq 0.$$ It implies that $\frac{\Im[g]}{\Re[g]}$ is a
decreasing function on $[0,\pi]$. Note that
$\Big(\frac{\Im[g]}{\Re[g]}\Big)\Big|_{x=\pi}=0$. As a result
$$\max_{x}\bigg|\frac{\Im[g]}{\Re[g]} \bigg|=\bigg(\frac{\Im[g]}{\Re[g]}\bigg)\bigg|_{x=0}=\tan(-\frac{\alpha}{2}\pi)\triangleq\rho_\alpha\quad \mbox{for}\quad x\in[0,\pi].$$
Similarly, when $x\in[-\pi,0]$, we have
\begin{align*}
g(\alpha,x)=&\big(2\sin(\frac{-x}{2}) \big)^\alpha \Big[\frac{\alpha}{2}\cos\big(\frac{\alpha}{2}(x+\pi)-x\big)
+\frac{2-\alpha}{2}\cos\big(\frac{\alpha}{2}(x+\pi)\big) \\
&{+}{\bf i}\Big( \frac{\alpha}{2}\sin\big(\frac{\alpha}{2}(x+\pi)-x\big)
+\frac{2-\alpha}{2}\sin\big(\frac{\alpha}{2}(x+\pi)\big) \Big) \Big],
\end{align*}
and
$$\bigg(\frac{\Im[g]}{\Re[g]}\bigg)_x=\frac{\frac{\alpha}{2}(1-\frac{\alpha}{2})(\alpha-1)\big(\cos(x)-1\big)}{\Big[\frac{\alpha}{2}\cos\big(\frac{\alpha}{2}(x+\pi)-x\big)
+\frac{2-\alpha}{2}\cos\big(\frac{\alpha}{2}(x+\pi)\big) \Big]^2}\leq 0.$$
 This implies that $\frac{\Im[g]}{\Re[g]}$ is still a decreasing function on $[-\pi,0]$.
 Note that $\Big(\frac{\Im[g]}{\Re[g]}\Big)\Big|_{x=-\pi}=0$.
 Therefore
$$\max_{x}\bigg|\frac{\Im[g]}{\Re[g]} \bigg|=-\bigg(\frac{\Im[g]}{\Re[g]}\bigg)\bigg|_{x=0}=\rho_\alpha\quad \mbox{for}\quad x\in[-\pi,0].$$
Consequently,
\begin{align*}
\min_{x}\frac{\Re[-g]}{|g(\alpha,x)|}=\frac{1}{\sqrt{1+
\rho_\alpha^2}}=|\cos(\frac{\alpha}{2}\pi)|\quad \mbox{for}\quad x\in[-\pi,\pi].
\end{align*}
This completes the proof.
\end{proof}

Referring to \cite{Jin}, we easily have the following lemma.
\begin{lemma}\label{JIN}
Let ${\bf u}=[u_1,u_2,\ldots,u_{M-1}]^T,{\bf v}=[v_1,v_2,\ldots,v_{M-1}]^T\in \mathbb{R}^{M-1}$. Then we have
$${\bf u}^TW_\alpha {\bf v}=\frac{1}{2\pi}\int_{-\pi}^\pi\sum_{k=1}^{M-1}{ u}_ke^{-{\bf i}kx}\sum_{k=1}^{M-1}{ v}_ke^{{\bf i}kx}g(\alpha,x)dx.$$
\end{lemma}

\begin{lemma}\label{w_k_property}(\cite{Deng_Secondorder})
The coefficients in \eqref{w_{k}(1,0)} satisfy the following properties for $1< \alpha\leq 2$,
\begin{equation}\nonumber
\left\{
\begin{array}{l}
w_0^{(\alpha)}=\frac{\alpha}{2},~w_1^{(\alpha)}<0,~w_2^{(\alpha)}=\frac{\alpha(\alpha^2+\alpha-4)}{4},\\
1\geq w_0^{(\alpha)}\geq w_3^{(\alpha)}\geq w_4^{(\alpha)}\geq \ldots \geq 0,\\
\sum_{k=0}^\infty w_k^{(\alpha)}=0,~ \sum_{k=0}^m w_k^{(\alpha)}<0,~m\geq 2.\\
\end{array}
\right.
\end{equation}
\end{lemma}
 One of the key steps in our analysis is to bound the field of values of
 $DWD$ by those of $W$, where $W=-W_\alpha-W_\alpha^T$ and $D$ is a diagonal matrix. This may not hold in general,
 in the followings, we established this relation for a special case.
 \begin{lemma}\label{key_inequality1}
 Denote $W=-W_\alpha-W_\alpha^T$.
 Suppose that $D$ is a diagonal matrix taking values $d_i$ of a function $d(x)$ at grid points $x_i~(1\le i\le M-1)$.
 For any ${\bf u}=[u_1,u_2,\ldots,u_{M-1}]^T$, we have
\begin{align}\label{key_inequality}
{\bf u}^TDWD{\bf u}\leq 2\max_i\{|d_i|^2\}{\bf u}^TW{\bf u},
\end{align}
if $d(x)$ is convex and $d(x)\geq0$, or $d(x)$ is concave and $d(x)\leq0$.
\end{lemma}
\begin{proof}
We have
\begin{eqnarray}\nonumber
W=\left[
    \begin{array}{ccccc}
       2\big|w_1^{(\alpha)}\big|& -(w_0^{(\alpha)}+w_2^{(\alpha)}) & -w_3^{(\alpha)} & \cdots &  -w_{M-1}^{(\alpha)} \\
       -(w_0^{(\alpha)}+w_2^{(\alpha)})& 2\big|w_1^{(\alpha)}\big| & -(w_0^{(\alpha)}+w_2^{(\alpha)}) & \ddots &  -w_{M-2}^{(\alpha)} \\
      -w_{3}^{(\alpha)} & -(w_0^{(\alpha)}+w_2^{(\alpha)}) & 2\big|w_1^{(\alpha)}\big| & \ddots & \vdots\\
      \vdots & \ddots & \ddots & \ddots & \vdots\\
      -w_{M-1}^{(\alpha)}  & \cdots  & \cdots & \cdots & 2\big|w_1^{(\alpha)}\big| \\
    \end{array}
\right],
\end{eqnarray}
\begin{eqnarray}\nonumber
D=\left[
    \begin{array}{cccc}
       d_{1}& & & \\
       & d_{2}&& \\
       && \ddots & \\
       &&&d_{M-1}\\
    \end{array}
\right],
\end{eqnarray}
and
{\small{\begin{eqnarray}\nonumber
DWD=\left[
    \begin{array}{ccccc}
       2d_1^2\big|w_1^{(\alpha)}\big|& -d_1d_2(w_0^{(\alpha)}+w_2^{(\alpha)}) & -d_1d_3w_3^{(\alpha)} & \cdots &  -d_1d_{M-1}w_{M-1}^{(\alpha)} \\
       -d_2d_1(w_0^{(\alpha)}+w_2^{(\alpha)})& 2d_2^2\big|w_1^{(\alpha)}\big| & -d_2d_3(w_0^{(\alpha)}+w_2^{(\alpha)}) & \ddots &  -d_2d_{M-1}w_{M-2}^{(\alpha)} \\
      -d_3d_1w_3^{(\alpha)} & -d_3d_2(w_0^{(\alpha)}+w_2^{(\alpha)}) & 2d_3^2\big|w_1^{(\alpha)}\big| & \ddots & \vdots\\
      \vdots & \ddots & \ddots & \ddots &\vdots\\
      -d_{M-1}d_1w_{M-1}^{(\alpha)}  & \cdots  & \cdots & \cdots & 2d_{M-1}^2\big|w_1^{(\alpha)}\big| \\
    \end{array}
\right].
\end{eqnarray}}}
Denote $d^\star={\max\limits_{i}}\{|d_i|^2\}$, then
$$2d^\star W-DWD=$$
{\scriptsize{\begin{eqnarray}\nonumber
\left[
    \begin{array}{ccccc}
       (4d^\star-2d_1^2)\big|w_1^{(\alpha)}\big|& (d_1d_2-2d^\star)(w_0^{(\alpha)}+w_2^{(\alpha)}) & (d_1d_3-2d^\star)w_3^{(\alpha)} & \cdots &  (d_1d_{M-1}-2d^\star)w_{M-1}^{(\alpha)} \\
       (d_2d_1-2d^\star)(w_0^{(\alpha)}+w_2^{(\alpha)})& (4d^\star-2d_2^2)\big|w_1^{(\alpha)}\big| & (d_2d_3-2d^\star)(w_0^{(\alpha)}+w_2^{(\alpha)}) & \ddots &  (d_2d_{M-1}-2d^\star)w_{M-2}^{(\alpha)} \\
      (d_3d_1-2d^\star)w_3^{(\alpha)} & (d_3d_2-2d^\star)(w_0^{(\alpha)}+w_2^{(\alpha)}) & (4d^\star-2d_3^2)\big|w_1^{(\alpha)}\big| & \ddots & \vdots\\
      \vdots & \ddots & \ddots & \ddots &\vdots\\
      (d_{M-1}d_1-2d^\star)w_{M-1}^{(\alpha)}  & \cdots  & \cdots & \cdots & (4d^\star-2d_{M-1}^2)\big|w_1^{(\alpha)}\big| \\
    \end{array}
\right].
\end{eqnarray}}}
From Lemma \ref{w_k_property}, we have the following properties:
\begin{align}\label{inequality_proof1}
w_0^{(\alpha)}+w_2^{(\alpha)}>0,\quad w_k^{(\alpha)}\geq0 ~(k\geq3)\quad \mbox{and}\quad \big|w_1^{(\alpha)}\big|>w_0^{(\alpha)}+w_2^{(\alpha)}+w_3^{(\alpha)}+\cdots+w_{M-1}^{(\alpha)}.
\end{align}
Note that $2d^\star W-DWD$ is a symmetric matrix with positive diagonal entries and nonpositive off-diagonal entries.
If the matrix is diagonally dominant, we can then conclude that it is positive definite by  Gershgorin circle theorem,
and  \eqref{key_inequality} follows.

 First we consider the top row of $2d^\star W-DWD$. If $d_i\geq0$ (i.e. $d(x)\geq0$) or $d_i\leq0$ (i.e. $d(x)\leq0$), it holds that
\begin{align}\label{inequality_proof2}
(4d^\star-2d_1^2)+(d_1d_k-2d^\star)=2d^\star-2d_1^2+d_1d_k\geq 0,\quad k\geq2.
\end{align}
Then we can get the first row of $2d^\star W-DWD$ is diagonally dominant from \eqref{inequality_proof1} and \eqref{inequality_proof2}.

 In general, for the $m$-th row with $2\leq m\leq\lfloor\frac{M}{2}\rfloor$, the sum of the magnitudes of off-diagonal entries is (if $m=2$, $w_m^{(\alpha)}$ is $(w_0^{(\alpha)}+w_2^{(\alpha)})$ in the following)
\begin{align*}
&(4d^\star-d_md_{m-1}-d_md_{m+1})(w_0^{(\alpha)}+w_2^{(\alpha)})+(4d^\star-d_md_{m-2}-d_md_{m+2})w_3^{(\alpha)}\\
&+\cdots+(4d^\star-d_md_1-d_md_{2m-1})w_m^{(\alpha)}+\sum_{k=m+1}^{M-m}(2d^\star-d_md_{k+1})w_k^{(\alpha)}.
\end{align*}
 On the other hand, the diagonal entry is
 $$(4d^\star-2d_m^2)\big|w_1^{(\alpha)}\big|>(4d^\star-2d_m^2)
 \left[(w_0^{(\alpha)}+w_2^{(\alpha)})+w_3^{(\alpha)}+\cdots+w_m^{(\alpha)}+\sum_{k=m+1}^{M-m}w_k^{(\alpha)}\right],$$
 Therefore, if  $d(x)$ is convex and nonnegative, or  concave and nonpositive, it follows that
 $$(4d^\star-2d_m^2)-(4d^\star-d_md_{m-j}-d_md_{m+j})=d_m(d_{m-j}+d_{m+j}-2d_m)\geq0,\quad 1\leq j\leq m-1,$$
 and
 $$(4d^\star-2d_m^2)-(2d^\star-d_md_{k+1})=2d^\star-2d_m^2+d_md_{k+1}\geq0,\quad m+1\leq k\leq M-m.$$
 These imply that the $m$-th $(2\leq m\leq\lfloor\frac{M}{2}\rfloor)$ row of matrix $2d^\star W-DWD$ is diagonally dominant.
 The analysis for $\lfloor\frac{M}{2}\rfloor+1\leq m\leq M-1$ is similar
 and we can then conclude that $2d^\star W-DWD$ is a diagonally dominant matrix
 if $d(x)$ satisfies the conditions stated in the lemma.
 This thus completes the proof.
\end{proof}

 We are now ready to introduce the condition mentioned in the beginning of this section.
 The following theorem shows that,
 if $d_+(x)$ and $d_-(x)$ satisfy this condition, the symmetric part ${\cal H}(A_{\alpha})$ of
 $A_\alpha=D_+^{-1}A$ (if $D_+$ invertible), or, $A_\alpha=D_-^{-1}A$ (if $D_-$ invertible)) is negative semidefinite. The condition depends on the ratio of $d_+(x)$ and $d_-(x)$. Denote ${\tilde d}(x)=\frac{d_-(x)}{d_+(x)}$, for simplicity, we only present the condition in terms of the assumption that
  $0\leq\kappa_{min}\leq {\tilde d}(x)\leq \kappa_{max}<\infty$ on $[x_L,x_R]$, that is, $A_\alpha=W_\alpha+{\tilde D}W^T_\alpha$ with ${\tilde D}=D_+^{-1}D_-$.
 Note that a similar statement holds for the assumption imposing on the ratio $d_+(x)/d_-(x)$. The theorem reads as:
\begin{theorem}\label{CONDITION}
 If the following condition holds
 \begin{equation}\label{cond}
 1+\kappa-\frac{\sqrt{2}(\kappa_{max}-\kappa_{min})}{\varsigma_\alpha}\geq0,
 \end{equation}
then the symmetric part ${\cal H}(A_\alpha)$ of $A_\alpha$ is negative semidefinite, where $\kappa=\kappa_{max}$ when ${\tilde d}(x)$ is concave, and $\kappa=\kappa_{min}$ when ${\tilde d}(x)$ is convex.
\end{theorem}
\begin{proof}
Denote $D={\tilde D}-\kappa I$, then $A_\alpha=W_\alpha+{\tilde D}W_\alpha^T=W_\alpha+\kappa W_\alpha^T+DW_\alpha^T$.
For any ${\bf u}=[u_1,u_2,\ldots,u_{M-1}]^T$, we have
\begin{align}\label{condproof1}
 2{\bf u}^T{\cal H}(A_\alpha){\bf u}
 =(1+\kappa){\bf u}^T\left(W_\alpha+W_\alpha^T\right){\bf u}+{\bf u}^T\left(DW_\alpha^T+W_\alpha D\right){\bf u}.
\end{align}
 Denote $u(x)=\sum_{k=1}^{M-1}{ u}_ke^{{\bf i}kx}$
 and $v(x)=\sum_{k=1}^{M-1}{(Du)}_ke^{{\bf i}kx}$, it follows by Lemma \ref{JIN}
 and Lemma \ref{key_inequality1} that
\begin{align}\label{condproof2}
 &{\bf  u}^TW{\bf  u}={\bf  u}^T\left(-W_\alpha-W_\alpha^T\right){\bf  u}
 =\frac1{\pi}\int_{-\pi}^\pi
  \Re[-g]|u(x)|^2dx,\\\label{condproof2-2}
 & {\bf  u}^TDWD{\bf  u}
 =\frac1{\pi}\int_{-\pi}^\pi
  \Re[-g]|v(x)|^2dx\leq \frac{2(\kappa_{max}-\kappa_{min})^2}{\pi}\int_{-\pi}^\pi
  \Re[-g]|u(x)|^2dx.
\end{align}
Then using Lemma \ref{JIN} again and applying Cauchy-Schwarz inequality, Theorem \ref{generatingfunction_rario} and \eqref{condproof2-2}, we get
   \begin{align}\nonumber
  \left|{\bf  u}^T\left(-DW_\alpha^T-W_\alpha D\right){\bf  u}\right|
  =&\frac1{4\pi}\Big|\int_{-\pi}^\pi
  \big[(-g^*-g)(v^*u+u^*v)+(-g^*+g)(v^*u-u^*v)\big] dx\Big|\\\nonumber
  {\le}&\frac1{\pi}\int_{-\pi}^\pi
  |g(\alpha,x)||v(x)||u(x)| dx\\\nonumber
  \leq& \frac1{\pi\varsigma_\alpha}\int_{-\pi}^\pi
  \Re[-g]|v(x)||u(x)| dx\\\nonumber
  \leq& \frac1{\pi\varsigma_\alpha}\sqrt{\int_{-\pi}^\pi
  \Re[-g]|v(x)|^2 dx}\sqrt{\int_{-\pi}^\pi
  \Re[-g]|u(x)|^2 dx}\\\label{condproof3}
  \le& \frac{\sqrt{2}(\kappa_{max}-\kappa_{min})}{\pi\varsigma_\alpha}\int_{-\pi}^\pi
  \Re[-g]|u(x)|^2 dx.
  \end{align}
Consequently,  \eqref{condproof1}--\eqref{condproof3} yield
\begin{align}\label{condproof4}
 -2{\bf u}^T{\cal H}(A_\alpha){\bf u}\geq
 \frac{1}{\pi}\left(1+\kappa-\frac{\sqrt{2}(\kappa_{max}-\kappa_{min})}{\varsigma_\alpha}\right)\int_{-\pi}^\pi\Re[-g]|u(x)|^2 dx.
\end{align}
It implies that ${\bf u}^T{\cal H}(A_\alpha){\bf u}\leq0$ if $1+\kappa-\frac{\sqrt{2}(\kappa_{max}-\kappa_{min})}{\varsigma_\alpha}\geq0$. Hence ${\cal H}(A_\alpha)$ is negative semidefinite under the condition \eqref{cond}.
\end{proof}
 \begin{remark} We note \eqref{cond} always holds when the diffusion coefficients are propositional to each other. This is in accordance with the results in \cite{Deng_Fourthorder}.
  \end{remark}


\begin{lemma}\rm{(\cite{MatrixAnalysis})}\label{bilinear_inequality}
Let symmetric matrix $H\in \mathbb{R}^{n\times n}$ with eigenvalues
$\lambda_1\ge \lambda_2\ge \ldots \ge \lambda_n$. Then for all $w\in
\mathbb{R}^{n}$,
$$\lambda_n w^Tw\leq w^THw\leq \lambda_1 w^Tw.$$
\end{lemma}

 Next we conclude the stability and convergence of scheme \eqref{Matrixform} by energy method. Once again,
 in the following two theorems, we only show the statements by assuming
  the condition is imposed on the ratio $\frac{d_-(x)}{d_+(x)}$,
  and remark that similar statements hold if the ratio $\frac{d_+(x)}{d_-(x)}$
  is under consideration.
\begin{theorem}\label{stability}
 If  the condition \eqref{cond} holds,
 the scheme \eqref{Matrixform}
 is stable and its solutions satisfy the following estimate
$$\left\|u^{n+1}\right\|_{\hat D}^2\leq e^{2T}\Big(\left\|\varphi\right\|_{\hat D}^2+2T\max_{0\leq k\leq n}\left\|f^{k+\frac12}\right\|_{\hat D}^2 \Big),\quad n=0,1,\ldots,N-1.$$
where $\|\cdot\|_{\hat D}$ is defined as $\|v\|_{\hat D}^2=h v^T{\hat D}v$ with ${\hat D}=D_+^{-1}$.
\end{theorem}
\begin{proof}
Multiplying $D_+^{-1}$ on the both sides of \eqref{Matrixform}, we get
\begin{align*}
(D_+^{-1}-\nu_{\tau, h, \alpha}A_\alpha)u^{n+1}=
 \left(D^{-1}_++\nu_{\tau, h,\alpha}A_\alpha\right)u^{n}
 +\tau D^{-1}_+f^{n+\frac{1}{2}},
\end{align*}
which is equivalent to
\begin{align}\label{stable1D1}
D_+^{-1}\left(u^{n+1}-u^{n} \right)-\nu_{\tau, h, \alpha}A_\alpha \left(u^{n+1}+u^{n} \right)=
\tau D^{-1}_+f^{n+\frac{1}{2}}.
\end{align}
Multiplying the both sides of \eqref{stable1D1} with $h\left(u^{n+1}+u^{n} \right)^T$, we have
\begin{align}\nonumber
h\left(u^{n+1}+u^{n} \right)^TD_+^{-1}\left(u^{n+1}-u^{n} \right)+h\nu_{\tau, h, \alpha}&\left(u^{n+1}+u^{n} \right)^T(-A_\alpha )\left(u^{n+1}+u^{n} \right)\\\label{stable1D2}
=&\tau h\left(u^{n+1}+u^{n} \right)^T D^{-1}_+f^{n+\frac{1}{2}}.
\end{align}
 Notice that $w^TQw=w^T{\cal H}(Q)w$ for any real $w$, $Q$, and ${\cal H}(-A_\alpha)$ is positive semidefinite  by Theorem \ref{CONDITION}.
  Therefore, the second term on the left hand side of \eqref{stable1D2} can be estimated as
 $$h\nu_{\tau, h, \alpha}\left(u^{n+1}+u^{n} \right)^T(-A_\alpha )\left(u^{n+1}+u^{n} \right)=h\nu_{\tau, h, \alpha}\left(u^{n+1}+u^{n} \right)^T{\cal H}(-A_\alpha )\left(u^{n+1}+u^{n} \right)\geq0.$$
 As a result
 \begin{align}\label{stable1D4}
 h(u^{n+1})^TD_+^{-1}u^{n+1}- h(u^{n})^TD^{-1}_+u^{n}\le
 \tau h(u^{n+1})^T D^{-1}_+f^{n+\frac{1}{2}}
 +\tau h(u^{n})^T D^{-1}_+f^{n+\frac{1}{2}}.
 \end{align}
Applying Cauchy-Schwarz inequality on the right hand side of \eqref{stable1D4}, we get
\begin{align*}
\left\|u^{n+1}\right\|_{\hat D}^2\leq\left\|u^{n}\right\|_{\hat D}^2+\frac{\tau}{2}\left\|u^{n+1}\right\|_{\hat D}^2+\frac{\tau}{2}\left\|u^{n}\right\|_{\hat D}^2+\tau\left\|f^{n+\frac12}\right\|_{\hat D}^2.
\end{align*}
That is
\begin{align}\label{stable1D5}
\left\|u^{n+1}\right\|_{\hat D}^2\leq\frac{2+\tau}{2-\tau}\left\|u^{n}\right\|_{\hat D}^2+\frac{2\tau}{2-\tau}\left\|f^{n+\frac12}\right\|_{\hat D}^2.
\end{align}
 Applying \eqref{stable1D5} iteratively for $n+1$ times, we have
\begin{align}\nonumber
\left\|u^{n+1}\right\|_{\hat D}^2\leq&\Big(\frac{2+\tau}{2-\tau}\Big)^{n+1}\left\|u^{0}\right\|_{\hat D}^2\\\label{stable1D6}
&+\frac{2\tau}{2-\tau}\Big[ 1+\frac{2+\tau}{2-\tau}+\Big(\frac{2+\tau}{2-\tau}\Big)^2+\ldots+\Big(\frac{2+\tau}{2-\tau}\Big)^{n}\Big]\max_{0\leq k\leq n}\left\|f^{k+\frac12}\right\|_{\hat D}^2.
\end{align}
When the time step size $\tau$ is sufficiently small $(\tau\leq1)$, we have
\begin{align}\label{stable1D7}
\Big(\frac{2+\tau}{2-\tau}\Big)^{n+1}=\Big(1+\frac{2\tau}{2-\tau}\Big)^{n+1}\leq(1+2\tau)^{n+1}\leq\big(1+\frac{2T}{N} \big)^N\overset{N\rightarrow\infty}{=}e^{2T},
\end{align}
and
\begin{align}\label{stable1D8}
\frac{2\tau}{2-\tau}\sum_{k=0}^n\Big(\frac{2+\tau}{2-\tau}\Big)^k\leq2\tau\sum_{k=0}^n\Big(\frac{2+\tau}{2-\tau}\Big)^{n+1}\leq
2\tau N\Big(\frac{2+\tau}{2-\tau}\Big)^{n+1}\leq 2Te^{2T}.
\end{align}
Consequently, substituting \eqref{stable1D7}, \eqref{stable1D8} into \eqref{stable1D6}, we obtain
$$\left\|u^{n+1}\right\|_{\hat D}^2\leq e^{2T}\Big(\left\|u^0\right\|_{\hat D}^2+2T\max_{0\leq k\leq n}\left\|f^{k+\frac12}\right\|_{\hat D}^2 \Big).$$
\end{proof}

\begin{theorem}\label{convergence1D}
Let $u(x,t)$ be the exact solution of \eqref{fde}, $u_i^n$ be the solutions of finite difference scheme \eqref{Matrixform}. Denote $e_i^n=u(x_i,t_n)-u_i^n$, $0\leq i\leq M$, $0\leq n\leq N$. If the condition \eqref{cond} holds, then there exists a positive constant $c_2$ such that
$$\|e^n\|^2\leq c_2(\tau^2+h^2)^2,$$
where $\|\cdot\|$ denotes the discrete $L^2$ norm, i.e. $\|v\|=\sqrt{hv^Tv}$.
\end{theorem}
\begin{proof}
Denote $e^n=[e_1^n,e_2^n,\ldots,e_{M-1}^n]^T$ and $R^{n+\frac{1}{2}}=[R_1^{n+\frac{1}{2}},R_2^{n+\frac{1}{2}},\ldots,R_{M-1}^{n+\frac{1}{2}}]^T$.
We can easily see that $e^n$ and $e_i^n$ satisfy the following error equation
\begin{align*}
&\left(I-\nu_{\tau,h,\alpha}A \right)e^{n+1}=\left(I+\nu_{\tau,h,\alpha}A \right)e^{n}+\tau R^{n+\frac{1}{2}}, \quad 0\leq n\leq N-1,\\
& e_0^n=e_M^n=0, \quad 1\leq n\leq N,  \qquad e_i^0=0, \quad 0\leq i\leq M.
\end{align*}
By Theorem \ref{stability}, we have
$$\left\|e^{n+1}\right\|_{\hat D}^2\leq 2T e^{2T}\max_{0\leq k\leq n}\left\|R^{k+\frac12}\right\|_{\hat D}^2 ,\quad n=0,1,\ldots,N-1.$$
Since $\hat{D}$ is a positive diagonal matrix and is bounded away from zero, by Lemma \ref{bilinear_inequality}, we can conclude that
$$\|e^{n+1}\|^2\leq c_2(\tau^2+h^2)^2,\quad n=0,1,\ldots,N-1.$$
\end{proof}

%
\section{Some Extensions}\label{secondorder2D}

 Note that the condition \eqref{cond} requires at least one of the coefficients $d_+(x)$ and $d_-(x)$ to be  positive because we multiply $A$ with $D_+^{-1}$ (or $D_-^{-1}$) so that proportional $d_+(x)$ and $d_-(x)$ are included as a special case
 in the derived condition.
 In fact, we can impose another condition without this restriction. Suppose that $0\leq\kappa_{min}^+\leq d_+(x)\leq \kappa_{max}^+<\infty$ and $0\leq\kappa_{min}^-\leq d_-(x)\leq \kappa_{max}^-<\infty$. Then we have the following theorem:
\begin{theorem}\label{CONDITION2}
 If the following condition holds
 \begin{equation}\label{cond2-1}
 \kappa^++\kappa^--\frac{\sqrt{2}(\kappa_{max}^++\kappa_{max}^--\kappa_{min}^+-\kappa_{min}^-)}{\varsigma_\alpha}>0,
 \end{equation}
then the symmetric part ${\cal H}(A)$ of $A$ in \eqref{Matrixform} is negative definite, where $\kappa^+=\kappa_{max}^+$ when $d_+(x)$ is concave, $\kappa^+=\kappa_{min}^+$ when $d_+(x)$ is convex, and, $\kappa^-=\kappa_{max}^-$ when $d_-(x)$ is concave, $\kappa^-=\kappa_{min}^-$ when $d_-(x)$ is convex.
\end{theorem}
\begin{proof}
 Denote ${\tilde D_+}=D_+-\kappa^+ I$ and ${\tilde D_-}=D_--\kappa^- I$.
 Then $A=\kappa^+W_\alpha+\kappa^-W_\alpha^T+{\tilde D_+}W_\alpha+{\tilde D_-}W_\alpha^T$.
For any ${\bf u}=[u_1,u_2,\ldots,u_{M-1}]^T$, we have
\begin{align*}
 2{\bf u}^T{\cal H}(A){\bf u}
 =(\kappa^++\kappa^-){\bf u}^T\left(W_\alpha+W_\alpha^T\right){\bf u}+{\bf u}^T\left({\tilde D_+}W_\alpha+W_\alpha^T {\tilde D_+}\right){\bf u}+{\bf u}^T\left({\tilde D_-}W_\alpha^T+W_\alpha {\tilde D_-}\right){\bf u}.
\end{align*}
Similar to the proof of Theorem \ref{CONDITION}, it follows that
\begin{align*}
\left|{\bf  u}^T\left(-{\tilde D_+}W_\alpha-W_\alpha^T {\tilde D_+}\right){\bf  u}\right|\le& \frac{\sqrt{2}(\kappa_{max}^+-\kappa_{min}^+)}{\pi\varsigma_\alpha}\int_{-\pi}^\pi
  \Re[-g]|u(x)|^2 dx,\\
\left|{\bf  u}^T\left(-{\tilde D_-}W_\alpha^T-W_\alpha {\tilde D_-}\right){\bf  u}\right|\le& \frac{\sqrt{2}(\kappa_{max}^--\kappa_{min}^-)}{\pi\varsigma_\alpha}\int_{-\pi}^\pi
  \Re[-g]|u(x)|^2 dx,
\end{align*}
and then
$$-2{\bf u}^T{\cal H}(A){\bf u}\geq
 \frac{1}{\pi}\left(\kappa^++\kappa^--\frac{\sqrt{2}(\kappa_{max}^++\kappa_{max}^--\kappa_{min}^+-\kappa_{min}^-)}{\varsigma_\alpha}\right)\int_{-\pi}^\pi\Re[-g]|u(x)|^2 dx.$$
Therefore, the theorem can be proved using arguments similar to those for Theorem \ref{CONDITION}.
\end{proof}

We next study the two-dimensional problem with variable coefficients:
 \begin{equation}\label{2Dfde}
\begin{array}{l}
\frac{\partial u(x,y, t)}{\partial t}=
 d_+(x)\,{_{x_L}}D_x^{\alpha}u(x,y,t)+d_-(x)\,{_x}D_{x_R}^{\alpha}u(x,y,t)+e_+(y)\,{_{y_L}}D_y^{\beta}u(x,y,t)\\
 \vspace{2mm}
 ~~~~~~~~~~~~~+e_-(y)\,{_y}D_{y_R}^{\beta}u(x,y,t)+f(x,y,t),\qquad (x,y)\in\Omega,~ t\in(0,T],\\
 \vspace{2mm}
  u(x,y,t)=0, \qquad (x,y)\in\partial\Omega,\quad t\in[0,T],\\
 u(x,y,0)=\varphi(x,y),\qquad (x,y)\in\bar\Omega.
 \end{array}
 \end{equation}
Here $1<\beta<2$, $\Omega=(x_L,x_R)\times(y_L,y_R)$, and $d_\pm(x)$, $e_\pm(y)$ are nonnegative functions.

 To state a finite difference scheme for \eqref{2Dfde},
 let $h_1=\frac{x_R-x_L}{M_1}$, $h_2=\frac{y_R-y_L}{M_2}$ and $\tau=\frac TN$ be the spatial and temporal step sizes respectively,
 where $M_1$, $M_2$ and $N$ are some given integers.
 For $i=0,1,\ldots,M_1, ~j=0,1,\ldots,M_2$, and $n=0,1,\ldots,N$,
 denote $x_i=ih_1$, $y_j=jh_2$ and $t_n=n\tau$.  Let $\bar{\Omega}_h=\{(x_i,y_j)|0\leq i\leq M_1, 0\leq j\leq M_2\}$, ${\Omega}_h={\bar{\Omega}_h}\cap{\Omega}$, $\partial{{\Omega}_h}={\bar{\Omega}_h}\cap{\partial{\Omega}}$.
  Furthermore, let $d_{+,i}=d_+(x_i)$, $d_{-,i}=d_-(x_i)$, and $e_{+,j}=e_+(y_j)$, $e_{-,j}=e_-(y_j)$.
   Let $u_{i,j}^n$ be the numerical approximation of $u(x_i,y_j,t_n)$, and $f_{i,j}^{n+\frac12}=f(x_i,y_j,t_{n+\frac12})$, $\varphi_{i,j}=\varphi(x_i,y_j)$. Then applying Crank-Nicolson method, the first equation of \eqref{2Dfde} can be discretized as
\begin{equation}\label{2d-scheme}
\begin{array}{l}\vspace{2mm}
\Big(\frac1{\tau}-\frac{1}{2}\delta_x^\alpha-\frac{1}{2}\delta_y^\beta \Big)u_{i,j}^{n+1}=
\Big(\frac1{\tau}+\frac{1}{2}\delta_x^\alpha+\frac{1}{2}\delta_y^\beta \Big)u_{i,j}^{n}+ f_{i,j}^{n+\frac12}
+{\cal O}(\tau^2+h_1^2+h_2^2), \\
~~~~~~~~~~~~~~~~~~~~~~~~~~~~~~~~~~~~~~~~~~~\quad\qquad\qquad(x_i,y_j)\in\Omega_h,~ 0\leq n\leq N-1,
\end{array}
\end{equation}
where
$$\delta_x^\alpha u_{i,j}^{n}=\frac{1}{h^{\alpha}}\left(d_{+,i}\sum_{k=0}^{{i}}w_{k}^{(\alpha)}u_{i-k+1,j}^{n}
 +d_{-,i}\sum_{k=0}^{{M_1-i}}w_{k}^{(\alpha)}u_{i+k-1,j}^{n}\right),$$
$$\delta_y^\beta u_{i,j}^{n}=\frac{1}{h^{\beta}}\left(e_{+,j}\sum_{k=0}^{{j}}w_{k}^{(\beta)}u_{i,j-k+1}^{n}
 +e_{-,j}\sum_{k=0}^{{M_2-j}}w_{k}^{(\beta)}u_{i,j+k-1}^{n}\right).$$

 Note that the size of \eqref{2d-scheme} is in general very large and we seek to improve efficiency by empolying ADI method.
 To this end, we add $\frac{\tau}{4}\delta_x^\alpha\delta_y^\beta(u_{i,j}^{n+1}-u_{i,j}^n)$ which is an ${\cal O}(\tau^2)$ term,
 to the left hand side of \eqref{2d-scheme}, and we obtain the following ADI approximation for \eqref{2Dfde}:
\begin{equation}\label{2dADI-scheme}
\begin{array}{l}\vspace{2mm}
\Big(1-\frac{\tau}{2}\delta_x^\alpha\Big)\Big(1-\frac{\tau}{2}\delta_y^\beta \Big)u_{i,j}^{n+1}=
\Big(1+\frac{\tau}{2}\delta_x^\alpha\Big)\Big(1+\frac{\tau}{2}\delta_y^\beta \Big)u_{i,j}^{n}+\tau f_{i,j}^{n+\frac12}+\tau R_{i,j}^{n+\frac12}, \\
~~~~~~~~~~~~~~~~~~~~~~~~~~~~~~~~~~~~~~~~~\quad\qquad\qquad(x_i,y_j)\in\Omega_h,~ 0\leq n\leq N-1,
\end{array}
\end{equation}
where $R_{i,j}^{n+\frac12}\leq c_3(\tau^2+h_1^2+h_2^2)$ for a positive constant $c_3$.

Take
$$u^{n}=[u^{n}_{1,1},u^{n}_{2,1},\ldots,u^{n}_{M_1-1,1},u^{n}_{1,2},\ldots,u^{n}_{M_1-1,2},
\ldots,u^{n}_{1,M_2-1},\ldots,u^{n}_{M_1-1,M_2-1}]^T,$$
$$f^{n}=[f^{n}_{1,1},f^{n}_{2,1},\ldots,f^{n}_{M_1-1,1},f^{n}_{1,2},\ldots,f^{n}_{M_1-1,2},
\ldots,f^{n}_{1,M_2-1},\ldots,f^{n}_{M_1-1,M_2-1}]^T,$$
and denote
$$A_x=\frac{1}{h_1^\alpha}\big[(I\otimes D_+)(I\otimes W_\alpha)+ (I\otimes D_-)(I\otimes W_\alpha^T)\big]
=\frac{1}{h_1^\alpha}I\otimes (D_+W_\alpha+D_-W_\alpha^T),$$
$$A_y=\frac{1}{h_2^\beta}\big[(E_+\otimes I)(W_\beta\otimes I)+ (E_-\otimes I)(W_\beta^T\otimes I)\big]
=\frac{1}{h_2^\beta}(E_+W_\beta+E_-W_\beta^T)\otimes I,$$
 where $I$ is the identity matrix and the symbol $\otimes$ denotes the Kronecker product,
 $W_\beta$, $W_\alpha$ are defined in \eqref{W},
 $D_{\pm}$ and $E_{\pm}$ are diagonal matrices taking the values of $d_{\pm}(x)$ and $e_{\pm}(y)$ at grid points respectively.

Therefore, omitting the small term $\tau R_{i,j}^{n+\frac12}$ in \eqref{2dADI-scheme}, the ADI scheme in matrix form for \eqref{2Dfde} can be given as:
\begin{align}\label{2dADI-matrix}
\big(I-\frac{\tau}{2}A_x\big)\big(I-\frac{\tau}{2}A_y \big)u^{n+1}=
\big(I+\frac{\tau}{2}A_x\big)\big(I+\frac{\tau}{2}A_y \big)u^{n}+\tau f^{n+\frac12}, \quad 0\leq n\leq N-1.
\end{align}
 Although one can easily extend arguments in  subsection \ref{1dsc} to study the scheme \eqref{2d-scheme},
 straight modification of these arguments do not work for the ADI scheme \eqref{2dADI-matrix}
 due to the extra term introduced by the method. We need further estimates for  our analysis.
As in Theorem \ref{generatingfunction_rario}, we denote
 $$\varsigma_\beta\triangleq\min_{x}\frac{\Re[-g(\beta,x)]}{|g(\beta,x)|}=|\cos(\frac{\beta}{2}\pi)|,$$
 where $g(\beta,x)$ is the generating function of matrix $W_\beta$. Then ${\cal H}(A_x)$
 and ${\cal H}(A_y)$ are negative definite respectively under the condition \eqref{cond2-1} and
 \begin{equation}\label{cond2-2}
 \chi^++\chi^--\frac{\sqrt{2}(\chi_{max}^++\chi_{max}^--\chi_{min}^+-\chi_{min}^-)}{\varsigma_\beta}>0,
 \end{equation}
 where $\chi_{max}^+,~\chi_{min}^+,~\chi_{max}^-,~\chi_{min}^-,~\chi^+,~\chi^-$, which have similar definitions with $\kappa_{max}^+,~\kappa_{min}^+,~\kappa_{max}^-,\\~\kappa_{min}^-,~\kappa^+,~\kappa^-$, are defined on the variable coefficients ${e_+(y)}$ and ${e_-(y)}$.

 Then, referring to the proof of Theorem 6 in \cite{Deng_Secondorder}, we have the following assertion on the stability of scheme \eqref{2dADI-matrix}.
 \begin{theorem}
 If the conditions \eqref{cond2-1} and \eqref{cond2-2} hold, the two-dimensional ADI scheme \eqref{2dADI-matrix} is stable.
 \end{theorem}

\section{Numerical Experiments}\label{numerical}
 In this section, we carry out numerical experiments for the proposed scheme \eqref{Matrixform} and \eqref{2dADI-matrix} to illustrate our theoretical statements. All our tests were done in MATLAB R2014a with a desktop computer (Dell optiplex 7020) having the following configuration: Intel(R) Core(TM) i7-4790 CPU 3.60GHz and 16.00G RAM. When computing the two-dimensional examples, we always take $h_1=h_2=h$.
 The $L^2$ norm errors between the exact and the numerical solutions
 $$E_2(h,\tau)=\max_{0\leq n\leq N}\|e^n\|,$$
 are shown in the following tables.
 Furthermore, the spatial convergence order, denoted by
 $$Rate1=\log_2\bigg(\dfrac{E_2(2h,\tau)}{E_2(h,\tau)}\bigg),$$
 for sufficiently small $\tau$, and the temporal convergence order, denoted
 by
 $$Rate2=\log_2\bigg(\dfrac{E_2(h,2\tau)}{E_2(h,\tau)}\bigg),$$
 when $h$ is sufficiently small, are reported.

\subsection{Accuracy Verification}
\begin{example}\label{ex1}
We consider the one-dimensional case \eqref{fde} for $x\in[0,1]$, $T=1$ with variable coefficients $d_+(x)=(x+2)^2$, $d_-(x)=5(x+2)^3$, and the forcing term
  \begin{align*}
  f(x,t)=&192x^3(1-x)^3t^2-2^6\left[\psi_3(x)-3\psi_4(x)+3\psi_5(x)-\psi_6(x) \right]t^3,
  \end{align*}
 where $\psi_k(x)=\frac{\Gamma(k+1)}{\Gamma(k+1-\alpha)}\left[ d_+(x)x^{k-\alpha}+d_-(x)(1-x)^{k-\alpha}\right]$ for $k=3,4,5,6$.
 Then the exact solution is $u(x,t)=2^6x^3(1-x)^3t^3$.
\end{example}
 $d_+(x)/d_-(x)=\frac1{5(x+2)}$ is convex on $[0,1]$. Table \ref{table1} lists the numerical results in spatial direction with fixed $\tau=\frac1{1000}$ for different choices of $\alpha$ which satisfy the condition \eqref{cond}. From the table, second order convergence of scheme \eqref{Matrixform} in space is apparent. The numerical results in temporal direction with fixed $h=\frac1{500}$ are recorded in Table \ref{table2}. One can see that they are  in accordance with the theoretical statement.

  \begin{table}[hbt!]
 \begin{center}
 \caption{Numerical results of scheme \eqref{Matrixform} for Example \ref{ex1} in spatial direction with $\tau=\frac{1}{1000}$.}\label{table1}
 \renewcommand{\arraystretch}{1.0}
 \def\temptablewidth{0.9\textwidth}
 {\rule{\temptablewidth}{0.9pt}}
 \begin{tabular*}{\temptablewidth}{@{\extracolsep{\fill}}ccccccc}
$h$ &\multicolumn{2}{c}{$\alpha=1.03$}&\multicolumn{2}{c}{$\alpha=1.1$}
 &\multicolumn{2}{c}{$\alpha=1.5$}\\
 \cline{2-3}\cline{4-5}\cline{6-7}
                &$E_2(h,\tau)$ &$Rate1$  &$E_2(h,\tau)$  &$Rate1$  &$E_2(h,\tau)$  &$Rate1$\\\hline
       $1/32$   &2.5354e-03    & $\ast$  &2.5622e-03     & $\ast$  &2.3497e-03     &$\ast$\\
       $1/64$   &6.2839e-04    &2.0125   &6.3474e-04     &2.0132   &5.8328e-04     &2.0102 \\
       $1/128$  &1.5658e-04    &2.0047   &1.5815e-04     &2.0049   &1.4542e-04     &2.0039 \\
       $1/256$  &3.8931e-05    &2.0079   &3.9321e-05     &2.0079   &3.6155e-05     &2.0080 \\
%
\end{tabular*}
{\rule{\temptablewidth}{0.9pt}}
\end{center}
\end{table}

 \begin{table}[hbt!]
 \begin{center}
 \caption{Numerical results of scheme \eqref{Matrixform} for Example \ref{ex1} in temporal direction with $h=\frac1{500}$ and $\alpha=1.5$.}
 \label{table2}
 \renewcommand{\arraystretch}{1.0}
 \def\temptablewidth{0.5\textwidth}
 {\rule{\temptablewidth}{0.9pt}}
 \begin{tabular*}{\temptablewidth}{@{\extracolsep{\fill}}ccc}
        $\tau$    & $E_2(h,\tau)$       & $Rate2$\\\hline
        $1/10$    & 4.3687e-03          & $\ast$\\
        $1/20$    & 1.0873e-03          & 2.0065\\
        $1/40$    & 2.6764e-04          & 2.0224\\
        $1/80$    & 6.3069e-05          & 2.0853\\
 \end{tabular*}
 {\rule{\temptablewidth}{0.9pt}}
 \end{center}
 \end{table}

\begin{example}\label{ex2}
 We consider the one-dimensional case \eqref{fde} once again for $x\in[0,1]$, $T=1$ with variable coefficients $d_+(x)=\cos\left[\frac{\pi}{12}(x+2)\right]$, $d_-(x)=\frac12(x-\frac12)^2$, and the forcing term $f(x,t)$ is choosing to such that $u(x,t)=2^6x^3(1-x)^3t^3$ is still the exact solution.
 \end{example}

  \begin{table}[hbt!]
 \begin{center}
 \caption{Numerical results of scheme \eqref{Matrixform} for Example \ref{ex2} in spatial direction with $\tau=\frac{1}{1000}$.}\label{table21}
 \renewcommand{\arraystretch}{0.96}
 \def\temptablewidth{0.9\textwidth}
 {\rule{\temptablewidth}{0.9pt}}
 \begin{tabular*}{\temptablewidth}{@{\extracolsep{\fill}}ccccccc}
$h$ &\multicolumn{2}{c}{$\alpha=1.17$}&\multicolumn{2}{c}{$\alpha=1.2$}
 &\multicolumn{2}{c}{$\alpha=1.5$}\\
 \cline{2-3}\cline{4-5}\cline{6-7}
                &$E_2(h,\tau)$ &$Rate1$  &$E_2(h,\tau)$  &$Rate1$  &$E_2(h,\tau)$  &$Rate1$\\\hline
       $1/32$   &2.2324e-03    & $\ast$  &2.2215e-03     & $\ast$  &2.0459e-03     &$\ast$\\
       $1/64$   &5.5622e-04    &2.0048   &5.5364e-04     &2.0045   &5.1092e-04     &2.0015 \\
       $1/128$  &1.3883e-04    &2.0032   &1.3821e-04     &2.0021   &1.2770e-04     &2.0003 \\
       $1/256$  &3.4547e-05    &2.0067   &3.4397e-05     &2.0065   &3.1803e-05     &2.0056 \\
\end{tabular*}
{\rule{\temptablewidth}{0.9pt}}
\end{center}
\end{table}

Note that $d_-(x)/d_+(x)$ is convex on $[0,1]$. We list the numerical results in spatial direction for different choices of $\alpha$ which fulfill the condition \eqref{cond} in Table \ref{table21}. The second order convergence rates in space are clearly shown in this table.

\begin{example}\label{ex3}
 We now consider the two-dimensional case \eqref{2Dfde} for $x,y\in[0,2]$, $T=1$ with variable coefficients $d_+(x)=\cos\left[\frac{\pi}{24}(x+4) \right]$, $d_-(x)=\sin\left[\frac{\pi}{24} (x+4) \right]$, $e_+(y)=\sin\left[\frac{\pi}{12} (y+6) \right]$, $e_-(y)=\frac18(y-1)^2$ and the forcing term
  \begin{align*}
  f(x,y,t)=&3x^4(2-x)^4y^4(2-y)^4t^2-\Big\{\big[16\psi_4(x)-32\psi_5(x)+24\psi_6(x)-8\psi_7(x)\\
  &+\psi_8(x) \big]y^4(2-y)^4+x^4(2-x)^4\big[16\phi_4(y)-32\phi_5(y)+24\phi_6(y)-8\phi_7(y)+\phi_8(y) \big]  \Big\}t^3,
  \end{align*}
 where $\psi_k(x)=\frac{\Gamma(k+1)}{\Gamma(k+1-\alpha)}\left[ d_+(x)x^{k-\alpha}+d_-(x)(2-x)^{k-\alpha}\right]$ and \\
 $\phi_k(y)=\frac{\Gamma(k+1)}{\Gamma(k+1-\beta)}\left[ e_+(y)y^{k-\beta}+e_-(y)(2-y)^{k-\beta}\right]$ for $k=4,5,6,7,8$.
Then the exact solution is $u(x,y,t)=x^4(2-x)^4y^4(2-y)^4t^3$.
\end{example}

  \begin{table}[hbt!]
 \begin{center}
 \caption{Numerical results of scheme \eqref{2dADI-matrix} for Example \ref{ex3} in spatial direction with $\tau=\frac{1}{1000}$.}\label{table3}
 \renewcommand{\arraystretch}{1.0}
 \def\temptablewidth{0.9\textwidth}
 {\rule{\temptablewidth}{0.9pt}}
 \begin{tabular*}{\temptablewidth}{@{\extracolsep{\fill}}ccccccc}
$h$ &\multicolumn{3}{c}{$(\alpha,\beta)=(1.22,1.31)$}&\multicolumn{3}{c}{$(\alpha,\beta)=(1.3,1.5)$}\\
 \cline{2-4}\cline{5-7}
                &$E_2(h,\tau)$&$Rate1$ &CPU   &$E_2(h,\tau)$&$Rate1$ & CPU\\\hline %
       $1/32$   &2.4858e-03   &$\ast$  &0.20  &2.5374e-03   &$\ast$  &0.19\\
       $1/64$   &6.2222e-04   &1.9982  &0.74  &6.3691e-04   &1.9942  &0.72\\
       $1/128$  &1.5559e-04   &1.9997  &2.27  &1.5941e-04   &1.9983  &2.31\\
       $1/256$  &3.8678e-05   &2.0081  &8.27  &3.9601e-05   &2.0091  &8.27
\end{tabular*}
{\rule{\temptablewidth}{0.9pt}}
\end{center}
\end{table}
Table \ref{table3} records the numerical results by applying scheme \eqref{2dADI-matrix} for Example \ref{ex3} in spatial direction with fixed $\tau=\frac1{1000}$, different choices $\alpha$ and $\beta$ which satisfy \eqref{cond2-1} and \eqref{cond2-2}, respectively, are taken.
 Here, `CPU' denotes the CPU times (seconds) in each step. It shows that the ADI scheme \eqref{2dADI-matrix},
 which computes the solutions within 10 seconds,  works efficiently. The numerical results show that the two-dimensional ADI scheme is stable, and the spatial second convergence rates are displayed in the table.

\subsection{Examples For Demonstrating Condition \eqref{cond}}\label{counterexample}
In this subsection, we test the behavior of the scheme \eqref{Matrixform}
when it is applied to some examples which do not satisfy the condition \eqref{cond}.

In the following, we replace the functions $d_+(x)$ and $d_-(x)$ given in Example \ref{ex1} with those listed below.
 The corresponding forcing terms $f(x,t)$ are chosen such that $u(x,t)=2^6x^3(1-x)^3t^3$ is still the exact solution of these examples.
\begin{example}\label{counter1}
$d_+(x)=-10000x(1-x)+2500$ and $d_-(x)=1$.
\end{example}
 \begin{example}\label{counter2}
 $d_+(x)=\cos(35\pi x)+1.01$ and $d_-(x)=\sin(35\pi x)+1.01$.
\end{example}
\begin{example}\label{counter3}
$$d_+(x)=\left\{\begin{array}{ll}
\frac{x}{10}, & x \in (0,\frac13),\\
\frac1{15},& x\in[\frac13,\frac23],\\
-\frac1{10}x+\frac1{10},\quad & x\in(\frac23,1);\end{array}\right. \quad \mbox{and} \quad d_-(x)=1.$$
\end{example}
 \begin{example}\label{counter4}
$$d_+(x)=\left\{\begin{array}{ll}
2x, & x \in (0,\frac13),\\
20,& x\in[\frac13,\frac23],\\
-2x+2,\quad & x\in(\frac23,1);\end{array}\right. \quad \mbox{and} \quad d_-(x)=1.$$
\end{example}
In Example \ref{counter1}, we note that ${\tilde d}(x)=\frac{d_+(x)}{d_-(x)}=d_+(x)\geq0$ and is convex but it does not satisfy the condition \eqref{cond}. In Example \ref{counter2}, $d_+(x)$, $d_-(x)$, and $\frac{d_+(x)}{d_-(x)}$ or $\frac{d_-(x)}{d_+(x)}$ are all  positive  functions, but
they are neither concave nor convex on $[0,1]$.
  Table \ref{table5-4-5} lists the numerical results for these two examples with $\alpha=1.005$ and $\tau$ being very small. From this table, we can see that the scheme \eqref{Matrixform} does not work as it is expected. Although the errors between the exact solution and numerical solution seem to become small for very small $h$ (comparing with those given in previous subsection), but
  it fails to be convergent with a certain order especially the second order which it is supposed to be.
  \begin{table}[hbt!]
 \begin{center}
 \caption{Numerical results of scheme \eqref{Matrixform} for Examples \ref{counter1} and \ref{counter2} with different $h$ and fixed $\tau=\frac{1}{30000}$, and $\alpha=1.005$.}\label{table5-4-5}
 \renewcommand{\arraystretch}{1.0}
 \def\temptablewidth{0.9\textwidth}
 {\rule{\temptablewidth}{0.9pt}}
 \begin{tabular*}{\temptablewidth}{@{\extracolsep{\fill}}ccccc}
          $h$ &\multicolumn{2}{c}{Example \ref{counter1}}&\multicolumn{2}{c}{Example \ref{counter2}}\\
             \cline{2-3}\cline{4-5}
                 &$E_2(h,\tau)$ &$Rate1$  &$E_2(h,\tau)$ &$Rate1$\\\hline
       $1/32$    &5.0314e+17    & $\ast$  &2.8551e+08    & $\ast$ \\
       $1/64$    &1.2203e+13    &15.3314  &3.1606e+04    &13.1410\\
       $1/128$   &4.7119e-03    &51.2018  &1.6826e+02    &7.5533\\
       $1/256$   &3.4166e+04    &-22.7897 &4.8162e+01    &1.8047\\
       $1/512$   &4.2817e-05    &29.5718  &5.9888e-03    &12.9733\\
       $1/1024$  &3.8552e-05    &0.1514   &7.3080e-04    &3.0347\\
       $1/2048$  &3.1526e-06    &3.6122   &3.1211e-07    &11.1932\\
       $1/4096$  &2.2212e-07    &3.8272   &3.1569e-08    &3.3055\\
\end{tabular*}
{\rule{\temptablewidth}{0.9pt}}
\end{center}
\end{table}
\begin{table}[hbt!]
 \begin{center}
 \caption{Numerical results of scheme \eqref{Matrixform} for Examples \ref{counter3} and \ref{counter4} with different $h$ and fixed $\tau=\frac{1}{30000}$.}\label{table5-6}
 \renewcommand{\arraystretch}{1.0}
 \def\temptablewidth{0.9\textwidth}
 {\rule{\temptablewidth}{0.9pt}}
 \begin{tabular*}{\temptablewidth}{@{\extracolsep{\fill}}ccccc}
          $h$ &\multicolumn{2}{c}{Example \ref{counter3} with $\alpha=1.07$}&\multicolumn{2}{c}{Example \ref{counter4} with $\alpha=1.01$}\\
             \cline{2-3}\cline{4-5}
                 &$E_2(h,\tau)$ &$Rate1$  &$E_2(h,\tau)$ &$Rate1$\\\hline
       $1/32$    &2.3481e-03    & $\ast$  &2.6635e-03    & $\ast$  \\
       $1/64$    &5.8405e-04    &2.0073   &1.1948e+04    &-22.0969 \\
       $1/128$   &1.4581e-04    &2.0020   &4.3134e-04    &24.7233  \\
       $1/256$   &3.6442e-05    &2.0005   &3.1704e+05    &-29.4532 \\
       $1/512$   &9.1094e-06    &2.0002   &1.2979e-05    &34.5077  \\
       $1/1024$  &2.2771e-06    &2.0001   &1.0014e-05    &0.3742   \\
       $1/2048$  &5.6909e-07    &2.0005   &8.4860e-07    &3.5608
\end{tabular*}
{\rule{\temptablewidth}{0.9pt}}
\end{center}
\end{table}

Examples \ref{counter3} and \ref{counter4} focus on testing the scheme \eqref{Matrixform} for discontinues variable coefficients.
For reference, we see that, in Table \ref{table5-6}, the scheme is stable and second order convergent when all conditions are satisfied by the discontinuous coefficient in Example \ref{counter3} with $\alpha=1.07$. However,
in Example \ref{counter4}, ${\tilde d}(x)=\frac{d_+(x)}{d_-(x)}=d_+(x)$ is concave but does not fulfill the condition \eqref{cond}.
From Table \ref{table5-6}, we observe  numerical results
similar to those of Examples \ref{counter1} and \ref{counter2}.

 Examples \ref{counter1}, \ref{counter2} and \ref{counter4} illustrate that the behavior of the second order scheme \eqref{Matrixform} are not predictable when solving problems with nonnegative variable coefficients which are not proportional to each other, especially when $\alpha$ is close to $1$ and the variable coefficients have steep slope or strong oscillations.
 Thus, for  second (or higher) order schemes of problem \eqref{fde} imposing conditions on the variable coefficients seems to be necessary.

\section{Concluding Remarks}\label{Concluding}
  We study finite difference schemes with a type of second order discretization for spatial fractional differential equations with variable coefficients.
  The approximation of fractional derivatives bases on the WSGD
 formula. Previous results (not only for those depending on the WSGD formula)
  concentrate on the cases  when the diffusion coefficients are proportional to each other.
  In this paper, we introduce a  condition under which the scheme can be analyzed theoretically.
  Stability and convergence of the scheme for one-dimensional problem is established.

\end{document}